\documentclass[12pt]{article}
\usepackage{amsmath,amssymb,amsthm}

\usepackage{enumerate}

\usepackage{mathrsfs}

\usepackage[all]{xy}
\usepackage{amscd}

\usepackage{enumitem}
\usepackage{geometry}
\usepackage{layout}
\newtheoremstyle{mystyle} %自分で作った定理スタイル
    {\topsep}				% 上部スペース
    {\topsep}				% 下部スペース
    {\normalfont}			% 本文フォント
    {} 					% インデント量
    {\bfseries} 			% 見出しフォント
    {\newline}				%見出し後の句読点
    {5pt plus 1pt minus 1pt}%見出し後のスペース
    {\underline{\thmname{#1}\thmnumber{#2}\thmnote{（#3）}}}	% 見出しの書式

\theoremstyle{definition}
\newtheorem{theorem}{Theorem}
\newtheorem{prop}[theorem]{Proposition}
\newtheorem{lem}[theorem]{Lemma}
\newtheorem{cor}[theorem]{Corollary}

\newtheorem{rem}[theorem]{Remark}
\numberwithin{theorem}{section} 	 % (sub)sectionと連動、定理3.1.1など
\numberwithin{equation}{section}		 % 等式はsectionと連動、式(3.5)

\newcommand \defeq{\overset{\text{def}}=}
\newcommand \ab{\operatorname{ab}}
\newcommand \Gal{\operatorname{Gal}}
\newcommand \isom {\overset \sim \rightarrow}

\def \Image{\operatorname{Im}}

\def \Aut{\operatorname{Aut}}
\def \tor{\operatorname{tor}}
\def \inf{\operatorname{inf}}

\def\p{\frak{p}}
\def\q{\frak{q}}

\def\Hom{\operatorname{Hom}}
\def\rank{\operatorname{rank}}

\newcommand\bZ{\mathbb Z}

\newcommand\bQ{\mathbb Q}

\newcommand\bC{\mathbb C}

\def \cs{\operatorname{cs}}

\def\tr{\operatorname{tr}}
\def\sup{\operatorname{sup}}
\def\inf{\operatorname{inf}}

\def\cO{{\mathcal O}}

\def \Iso{\operatorname{Iso}}
\def \OutIso{\operatorname{OutIso}}
\def \Inn{\operatorname{Inn}}
\def \Out{\operatorname{Out}}

\begin{document}

\renewcommand{\thesection}{\arabic{section}}
%%\renewcommand{\thesubsection}{\thesection.\arabic{subsection}}
%%\renewcommand{\thesection}{\S\;\arabic{section}}
%%色々付随して変わる

\renewcommand\thefootnote{*\arabic{footnote}}

\title{Isomorphisms of Galois groups of number fields with restricted ramification}
\author{ \textsc{Ryoji Shimizu}\footnote{RIMS, Kyoto University, Kyoto
606-8502, Japan.
%%\newline 
e-mail: \texttt{shimizur@kurims.kyoto-u.ac.jp}}}
\date{
%%2020年4月10日
%%January 4, 2021
}
%%\classification{Primary 11R32; Secondary 11S15.}
%%11R32 Galois theory 11S15 Ramification and extension theory
\maketitle

%%\runninghead{Neukirch-Uchida with restricted ramification}

\begin{abstract}
Let $K$ be a number field and $S$ a set of primes of $K$.
We write $K_S/K$ for the maximal extension of $K$ 
%%which is 
unramified outside $S$ and $G_{K,S}$ for its Galois group. 
In this paper, we 
%%mainly study the following question: 
answer the following question under some assumptions: 
``For $i=1,2$, 
let $K_i$ be a number field, $S_i$ a (sufficiently large) set of primes of $K_i$ and $\sigma :G_{K_1,S_1}\isom G_{K_2,S_2}$ an isomorphism.
%%For any $K_i$, any $S_i$ and any $\sigma$, is 
Is $\sigma$ induced by a unique isomorphism between $K_{1,S_1}/K_1$ and $K_{2,S_2}/K_2$?''
%%The main theorem answer this under some assumptions.
%%The assumptions include: 
%%%%$K_i$ is Galois over $\bQ$ for $i = 1,2$; $K_1$ is totally imaginary; 
Here the main assumption is 
about the Dirichlet density of $S_i$.
%%that the Dirichlet density of $S_i$ is not zero for at least one $i$.
%%; and so on.
\end{abstract}

\tableofcontents

\section{Introduction}
Let $K$ be a number field and $S$ a set of primes of $K$. We write $K_S/K$ for the maximal extension of $K$ unramified outside $S$ and $G_{K,S}$ for its Galois group. 

%%The ``weakly bi-anabelian'' result 
%%known as the Neukirch-Uchida theorem, 
The Neukirch-Uchida theorem, 
which is one of the most important results in anabelian geometry, states that 
%%the isomorphisms of the absolute Galois groups of number fields arise functorially from unique isomorphisms of fields.
if the absolute Galois groups of number fields are isomorphic, then the number fields are isomorphic (cf. \cite{Neukirch} and {\cite{Uchida}}).
%%{\cite[COROLLARY 2]{Uchida}}).
Moreover, 
Uchida also proved that the isomorphisms of the absolute Galois groups of number fields arise functorially from unique isomorphisms of fields (cf. {\cite[THEOREM]{Uchida}}).
On the other hand, 
Ivanov in 
\cite{Ivanov2}, \cite{Ivanov} 
and 
\cite{Ivanov3}, and succeedingly the author in \cite{Shimizu}, 
studied 
a generalization of the Neukirch-Uchida theorem where one replaces the absolute Galois groups by the Galois groups of the maximal extensions with restricted ramification. 
These results prompt the following natural question (cf. {\cite[(12.3.4) Question]{NSW}}): 

{\it 
For $i=1,2$, 
let $K_i$ be a number field, $S_i$ a (sufficiently large) set of primes of $K_i$ and $\sigma :G_{K_1,S_1}\isom G_{K_2,S_2}$ an isomorphism.
%%let $K_i$ be a number field, $S_i$ a set of primes of $K_i$, $\sigma :G_{K_1,S_1}\isom G_{K_2,S_2}$ an isomorphism and $\overline{\sigma}$ the isomorphism between some quotients of $G_{K_1,S_1}$ and $G_{K_2,S_2}$ induced by $\sigma$.
%%Let $\overline{\sigma}$ be 
%%For any $K_i$ and any $S_i$, is $\sigma$ (or the isomorphism between some quotients of $G_{K_1,S_1}$ and $G_{K_2,S_2}$ induced by $\sigma$) induced by an (a unique) isomorphism between (subfields of) $K_{1,S_1}$ and $K_{2,S_2}$?
%%For any $K_i$, any $S_i$ and any $\sigma$, 
%%$\overline{\sigma}$, 
%%is 
Is 
$\sigma$ 
%%$\overline{\sigma}$ 
induced by a unique 
%%an (a unique) 
isomorphism between 
$K_{1,S_1}/K_1$ and $K_{2,S_2}/K_2$?
%%the corresponding subfields of $K_{1,S_1}$ and $K_{2,S_2}$?
}

In this paper, to approach this question, 
we 
mainly 
refine arguments in \cite{Shimizu}.

In \S 1, we prove the faithfulness of the Galois action on the Galois group of the maximal multiple $\bZ_l$-extension.
By this result, we obtain the ``uniqueness" in question under a mild assumption.

In \S 2, we 
develop a way, 
%%whose arguements are 
based on \cite{Shimizu}, \S 3, 
to show that 
isomorphisms of Galois groups are induced by 
(unique) field isomorphisms 
%%using the results obtained so far (Theorem \ref{3.4}).
under some assumptions, for example about the Dirichlet density.
%%assuming the good local correspondence holds.
%%The arguements are essentially based on \cite{Shimizu}, \S 3, and sharpening
Then, by using this, we prove the main result (more precisely, see Theorem \ref{2.5}): 
{\it 
For $i=1,2$, 
let $K_i$ be a number field, $S_i$ a set of primes of $K_i$ and $\sigma :G_{K_1,S_1}\isom G_{K_2,S_2}$ an isomorphism.
Assume that for $i=1,2$ and for any finite Galois subextension $L_i$ of $K_{i,{S_i}}/K_i$, 
%%$\delta(P_{S_i,f} \cap \cs(L_i/\bQ)) \neq 0$ 
the Dirichlet density of $P_{S_i,f} \cap \cs(L_i/\bQ)$ is not zero
(cf. Notations).
%%Let $\sigma :G_{K_1,S_1}\isom G_{K_2,S_2}$ be an isomorphism.
Then there exists a unique isomorphism $\tau:K_{2,{S_2}} \isom K_{1,{S_1}}$ such that $K_1=\tau(K_2)$ and $\sigma$ coincides with the isomorphism induced by $\tau$.
}

In \S 3, we see some applications of the main theorem, for example about the set of outer isomorphisms of $G_{K,S}$ (Corollary \ref{3.5}).
%%幾何的な言い換えや、Out

In \S 4, we study intersections of decomposition groups in $G_{K,S}$. 
%%at primes in $S_f(K_S)$.
%%The results in \S 4 are not used in the proof of the main theorem. However, the 
The results not only are interesting in themselves, 
but also give an alternative 
%%another 
proof of the ``uniqueness" in question. 
%%an a little weaker version of Proposition \ref{0.2}.

\section*{Acknowledgements}
The author 
would like to thank
Professor Akio Tamagawa 
for helpful 
%%discussions and 
advices 
and 
carefully reading preliminary versions of 
the present paper. 

The author 
would also like to thank
Professor Florian Pop 
for comments on the proofs of Lemma \ref{1.2} and Lemma \ref{1.3} in a former version of this paper.
%% of the present paper.

This work was supported by JSPS KAKENHI Grant Number 21J11879.

\section*{Notations}
\begin{itemize}[leftmargin=*]
\item[$\bullet$]
Given a set $A$ we write $\# A$ for its cardinality.

\item[$\bullet$]
For a profinite group $G$, let $\overline{[G, G]}$ be the closed subgroup of $G$ which is (topologically) generated by the commutators in $G$. We write $G^{\ab} \defeq G/\overline{[G,G]}$
for the maximal abelian quotient of $G$.

\item[$\bullet$]
For a profinite group $G$, we say that $G$ is topologically infinitely generated if $G$ is not topologically finitely generated.

\item[$\bullet$]
Given a profinite group $G$ 
and a prime number $l$, we write $G^{(l)}$ for the maximal
pro-$l$ quotient of $G$.

\item
Given a Galois extension $L/K$, we write $G(L/K)$ for its Galois group $\Gal(L/K)$.
Given a field $K$, we write $\overline{K}$ for a separable closure of $K$, 
%%Ksep for the maximalseparable extension of K contained in K, 
and $G_K$ for the absolute Galois group $G(\overline{K}/K)$ of $K$.

%%\item
%%Given a field $K$, we write $K^{\ab}$ for the maximal abelian subextension of $\overline{K}/K$, which corresponds to the quotient $G_K \twoheadrightarrow G_K^{\ab}$.

%%\item
%%Given a field $K$ and a prime number $l$, we write $K^{(l)}$ for the maximal pro-$l$ subextension of $\overline{K}/K$, which corresponds to the quotient $G_K \twoheadrightarrow G_K^{(l)}$.

\item
A number field is a finite extension %field
 of the field of rational numbers $\bQ$. 
For an (a possibly infinite) algebraic extension $F$ of $\bQ$, 
%%For a number field $K$, we write $\widetilde{K}$ for the Galois closure of $K/\bQ$.
we write 
%%$\widetilde{F}$ for the Galois closure of $F/\bQ$, 
$P=P_F$ 
%%(resp. $\Primes_F^{\na}$) 
for the set of primes 
%%(resp. nonarchimedean primes) 
of $F$, 
$P_\infty=P_{F,\infty}$ for the set of archimedean primes of $F$, 
%%$r_\bC(F)$ (resp. $r_\bR(F)$) for the number of complex (resp. real) primes of $F$, 
and, 
for a prime number $l$, 
$P_l=P_{F,l}$ for the set of nonarchimedean primes of $F$ above $l$. 
Further, for a set of primes $S \subset P_F$, 
we set $S_f\defeq S \setminus P_\infty$, 
$P_S \defeq \{ p \in P_{\bQ} \mid P_{F,p} \subset S \}$.
%%\item
For $\bQ \subset F \subset F' \subset \overline{\bQ}$, 
we write 
%%$\cs(F'/F)$ (resp. $\Ram(F'/F)$) for the set of nonarchimedean primes of $F$ which split completely (resp. ramified) in $F'/F$.
%%$\cs(F'/F) \defeq \{ \p \in \Primes_F^{\na} \mid aaaaaaa \}$.
%%
$S(F')$ for the set of primes of $F'$ above the primes in $S$: 
$S(F') \defeq \{ \p \in P_{F'} \mid \p|_F \in S \}$.
%%
%%ここでいいでしょう
For convenience, we consider 
%define?say?
that $F'/F$ is ramified at a complex prime of $F'$ if it is above a real prime of $F$.
We write $F_S/F$ for the maximal extension of $F$ unramified outside $S$ and $G_{F,S}$ for its Galois group.
When $P_\infty \subset S$, we set 
$\mathcal O_{F,S} \defeq \{ a \in F \mid 
%%v_\p(a) \geq 0 
|a|_{\p} \leq 1
\text{ for all } \p \notin S \}$, where 
%%$v_\p$ is the (normalized) exponential valuation associated to $\p$.
$|\ |_{\p}$ is an absolute valuation associated to $\p$.

\item
For $i=1,2$, 
let $A_i$ be a (commutative) ring.
Write $\Iso(A_2, A_1)$ for the set of ring isomorphisms from $A_2$ to $A_1$.
For $i=1,2$, 
let $B_i$ be a ring containing $A_i$.
Write $\Iso(B_2/A_2, B_1/A_1) \defeq \{ \tau\in \Iso(B_2, B_1) \mid \tau(A_2)=A_1 \}$.
For $i=1,2$, 
let $K_i$ be a number field, $L_i/K_i$ an algebraic extension,  and $S_i$ a set of primes of $K_i$.
Write 
\begin{equation*}
\begin{split}
&\Iso((K_2,S_2), (K_1,S_1)) \defeq \{ \tau\in \Iso(K_2, K_1) \mid \text{$\tau$ induces a bijection between $S_2$ and $S_1$}\},\\
&\Iso((L_2/K_2,S_2), (L_1/K_1,S_1)) \defeq 
%%\{ \tau\in \Iso(L_2/K_2, L_1/K_1) \mid \text{$\tau$ induces a bijection between $S_2(L_2)$ and $S_1(L_1)$}\}
\left\{ \tau\in \Iso(L_2/K_2, L_1/K_1) \left|
\begin{array}{l}
\text{
$\tau$ induces a bijection between}\\ 
\text{$S_2(L_2)$ and $S_1(L_1)$}
\end{array}
\right.\right\}.
\end{split}
\end{equation*}
For a number field $K$ and a set of primes $S$ of $K$, write 
$\Aut(K,S) \defeq \Iso((K,S), (K,S))$.

\item
For a profinite group $G$, write $\Inn(G)$ for the set of inner automorphisms of $G$.
For $i=1,2$, 
let $G_i$ be a profinite group.
Write $\Iso(G_1, G_2)$ for the set of isomorphisms of profinite groups from $G_1$ to $G_2$.
Note that the group $\Inn(G_2)$ acts on $\Iso(G_1, G_2)$ by the rule $\sigma(\phi) \defeq \sigma\circ\phi$ for $\sigma\in\Inn(G)$ and $\phi\in\Iso(G_1, G_2)$.
We call 
$\OutIso(G_1, G_2) \defeq \Iso(G_1, G_2)/\Inn(G_2)$
the set of outer isomorphisms from $G_1$ to $G_2$.
Write $\Out(G) \defeq \OutIso(G, G)$.

\item
Given an algebraic extension $K$ of $\bQ$ and $\p \in P_{K,f}$, 
we write $\kappa(\p)$ for the residue field at $\p$. 
%%Let $K$ be an algebraic extension of $\bQ$ and $\p$ be an element of $\Primes_{K}^{\na}$.
When $K$ is a number field, we write $K_\p$ for the completion of $K$ at $\p$, 
and, in general, we write $K_\p$ for the union of $K'_{\p|_{K'}}$ for finite subextensions $K'/\bQ$ of $K/\bQ$.

\item
Let $L/K$ be a finite extension of number fields and $\q \in P_{L,f}$, 
and set $\p = \q|_{K}$.
We write $f_{\q,L/K} \defeq [\kappa(\q):\kappa(\p)]$.
%%We also write $f_{\p,L/K} = f_{\q,L/K}$, when no confusion arises.
We write $\cs(L/K)$ 
%%(resp. $\Ram(L/K)$) 
for the set of nonarchimedean primes of $K$
which split completely 
%%(resp. are ramified) 
in $L/K$.

\item
Let $K$ be a number field and $\p \in P_{K,f}$, 
and set $p = \p|_{\bQ}$.
%%Given a number field $K$ and $\p \in P_{K,f}$ with $p = \p|_{\bQ}$, $K_{\p}/\bQ_p$ is a finite extension.
Define the residual degree (resp. the local degree) of $\p$ 
as 
 $f_{\p,K/\bQ}$ 
(resp. $[K_\p:\bQ_p]$).
We set $\frak{N}(\p) \defeq \# \kappa(\p) = p^{f_{\p,K/\bQ}}$.
%%the norm

\item
For a number field $K$ and a set of primes $S \subset P_K$, 
we set 
$$
\delta_{\sup}(S) \defeq  
\limsup_{s \to 1+0} \frac{\sum_{\p \in S_f} \frak{N}(\p)^{-s}}{\log{\frac{1}{s-1}}}
,\ 
\delta_{\inf}(S) \defeq  
\liminf_{s \to 1+0} \frac{\sum_{\p \in S_f} \frak{N}(\p)^{-s}}{\log{\frac{1}{s-1}}}
$$
and 
if $\delta_{\sup}(S) = \delta_{\inf}(S)$, 
then write $\delta(S)$ (the Dirichlet density of $S$) for them.
%%For convenience, 
The term 
``$\delta(S) \neq 0$''
will always mean that 
$S$ has positive Dirichlet density 
or 
$S$ does not have Dirichlet density.
Note that 
$\delta(S) \neq 0$ 
if and only if 
$\delta_{\sup}(S) > 0$.

\item
For $\bQ \subset F \subset F' \subset 
%%F'' \subset 
\overline{\bQ}$ with $F'/F$ Galois, $\q \in P_{F',f}$ and $\p = \q|_F$, 
write $
%%G(F'/F)_{\q}=
D_{\q}(F'/F) \subset G(F'/F)$ for the decomposition group (i.e. the stabilizer) of $\q$ in $G(F'/F)$. 
We sometimes write $D_{\q} = D_{\q}(F'/F)$, when no confusion arises.
%%Further, we also write $D_{\p} 
%%= G(F'/F)_{\p} 
%%= D_{\p}(F'/F) = D_{\q}(F'/F)$, when no confusion arises. Note that $D_{\p}$ is only defined up to conjugation.
There exists a canonical isomorphism $D_{\q}(F'/F) \simeq G(F'_\q/F_\p)$,
and 
we will identify $D_{\q}(F'/F)$ with $G(F'_\q/F_\p)$ via this isomorphism.

\item
Let $p$ be a prime number.
%%a 
A 
$p$-adic field is a finite extension %field
 of the field of $p$-adic numbers $\bQ_p$. 
Let $\kappa$ be 
%%an (a possibly infinite) algebraic extension of $\bQ_p$.
a $p$-adic field.
We write $V_{\kappa}$ 
%%(resp. $I_{\kappa}$) 
for the ramification 
%%(resp. inertia) 
subgroup of $G_\kappa$, 
%%and $\kappa^{\tr}$ 
%%(resp. $\kappa^{\ur}$) 
%%for the subextension of $\overline{\kappa}/\kappa$ corresponding to $V_{\kappa}$, 
%%(resp. $I_{\kappa}$), 
and set 
$G_\kappa^{\tr} \defeq G_\kappa/V_{\kappa}$.
%% and $G_\kappa^{\ur} \defeq G_\kappa/I_{\kappa}$.
%%
%%
Let $\lambda/\kappa$ be a Galois extension.
We say that $G(\lambda/\kappa)$ is full if $\lambda$ is algebraically closed. 
%%We write $I(\lambda/\kappa)$ for the inertia subgroup of $G(\lambda/\kappa)$. When $\kappa$ is a $p$-adic field, we say that an element of $G(\lambda/\kappa)$ is a Frobenius lift if its image under $G(\lambda/\kappa) \twoheadrightarrow G(\lambda/\kappa)/I(\lambda/\kappa)$ is equal to the Frobenius element.

\item
Given an abelian group $A$, 
%% and a prime number $l$, 
we write $A_{\tor}$ for the torsion subgroup of $A$.
%%, and $A[l^\infty]$ for the $l$-primary part of $A_{\tor}$.

\item
Given an abelian profinite group $A$,
we write $\overline{A_{\tor}}$ for the closure in $A$ of $A_{\tor}$,
%%要る？
 and set $A^{/\tor} \defeq A/\overline{A_{\tor}}$.

\item
Given a field $K$, we write $\mu (K)$ for the group consisting of the roots of unity in $K$.
For $n \in \bZ_{>0}$ 
not divisible by the characteristic of $K$, 
we write $\mu_n=\mu_n (\overline{K}) \subset \mu (\overline{K})$ for the subgroup of order $n$. 
%%if it exists.
For a prime number $l$ distinct from the characteristic of $K$, 
we set 
$\mu_{l^\infty} \defeq \bigcup_{n \in \bZ_{>0}} \mu_{l^n}(\overline{K}) \subset \mu (\overline{K})$.

%%\item
%%Let $l$ be a prime number. We set $\ltilde \defeq l$\ (resp.\ $\ltilde \defeq 4$) for $l\neq2$\ (resp.\ $l=2$).

\end{itemize}

%%%%%%%%%%%%%%%%%%%%%%%%%%%%%%%%%%%%%%%%%%%%%%%%   セクション１　%%%%%%%%%%%%%%%%%%%%%%%%%%%%%%%%%%

\section{The faithfulness of the Galois action 
%%of the Galois groups on $\Gamma_{L, l}$
on the Galois group of the maximal multiple $\bZ_l$-extension}
%% and its application}

In the rest of this 
%%section, 
paper, 
%%Let 
let $K$ be a number field and
$S \subset P_K$ a set of primes of $K$.

Let $l$ be a prime number.
We set 
$%G_{K}\twoheadrightarrow 
\Gamma_K = \Gamma_{K, l} \defeq G_{K}^{\ab,(l),/\tor}$.
Then $\Gamma_K$ is a free $\bZ_l$-module.
Set $r_l(K) \defeq \rank_{\bZ_l}\Gamma_K$, 
and write $K^{(\infty)}=K^{(\infty,l)}$ for the extension of $K$ corresponding to $\Gamma_K$.
%% with this surjection.
$K^{(\infty,l)}/K$ is unramified outside $P_l$ by class field theory, and hence, if $P_l \subset S$, $\Gamma_{K, l} = G_{K,S}^{\ab,(l),/\tor}$.
%%We write $\fd_l(K)(\geq 0)$ for the Leopoldt defect. Then, by 
By 
\cite{NSW}, (10.3.20) Proposition, 
we have $r_\bC(K)+1 \leq r_l(K) \leq [K:\bQ]$, 
%%$r_l(K) = r_\bC(K)+1+\fd_l(K) \leq [K:\bQ]$, 
where $r_\bC(K)$ is the number of complex primes of $K$.
For a finite extension $L/K$, 
%%set the canonical homomorphisms $\pi_{L/K} = \pi_{L/K, l}:\Gamma_{L,l} \to \Gamma_{K,l}$.
%%denote the canonical homomorphism: $\Gamma_{L,l} \to \Gamma_{K,l}$ by $\pi_{L/K} = \pi_{L/K, l}$.
define $\pi_{L/K} = \pi_{L/K, l}$ to be the canonical homomorphism: $\Gamma_{L,l} \to \Gamma_{K,l}$.
We write $\Hom_{cts}(\Gamma_{K, l}, \bZ_l)$ 
%%(\simeq \Hom_{cts}(G_{K}, \bZ_l))$ 
for the set of continuous homomorphisms from $\Gamma_{K, l}$ to $\bZ_l$, 
where $\bZ_l$ is equipped with the profinite topology.
Then $\Hom_{cts}(\Gamma_{K, l}, \bZ_l) = \Hom_{\bZ_l}(\Gamma_{K, l}, \bZ_l) \simeq \bZ_l^{r_l(K)}$.

\begin{lem}\label{2.1}
Let $K$ be a number field, $L$ a finite Galois extension of $K$, and $l$ a prime number.
Assume that $L$ has a complex prime. 
Then the canonical action of $G(L/K)$ on $\Gamma_{L, l}$ induced by conjugation is faithful.
\end{lem}
\begin{proof}
Write $K'$ for the subextension of $L/K$ corresponding to the kernel of the homomorphism: $G(L/K) \to \Aut(\Gamma_{L, l})$ induced by the canonical action.
By \cite[Lemma 3.1]{Shimizu}, we have $r_l(L) = \rank_{\bZ_l} \Hom_{cts}(\Gamma_{L, l}, \bZ_l) = \rank_{\bZ_l} \Hom_{cts}(\Gamma_{L, l}, \bZ_l)^{G(L/K')} = r_l(K')$.
It follows from \cite[Lemma 3.2]{Shimizu} that $L = K'$.
\end{proof}

\begin{prop}\label{0.2}
Assume 
$\# P_{S,f} \geq 1$ and that $K_S$ has a complex prime.
Let $\tau \in \Aut(K_S)$. 
Assume $\tau(K)=K$ and that the automorphism of $G_{K,S}$ induced by the conjugation action of $\tau$ 
%%by $\tau$ by conjugation 
is trivial.
Then $\tau$ is trivial.
\end{prop}
\begin{proof}
%%Write $\tau^\ast$ for the automorphism of $G_{K,S}$ induced by $\tau$. $G_{K,S} \subset \Aut(K_S)$.
Write $K_0$ for the $\Aut(K_S)$-invariant subfield of $K_S$. 
Then $K_S/K_0$ is Galois.
Take $l \in P_{S,f}$.
Let $L$ be any finite Galois subextension of $K_{S}/K_0$ 
containing $K$, and 
having a complex prime.
Then, by assumption, the action of $\tau|_{L}$ on $\Gamma_{L, l}$ 
%%induced by conjugation 
is trivial.
Therefore, by Lemma \ref{2.1}, $\tau|_{L}$ is trivial.
Thus, $\tau$ is trivial.
\end{proof}

\begin{cor}\label{1.6}
Assume 
%%$P_{\infty} \subset S$ and that $\# P_{S,f} \geq 1$.
$\# P_{S,f} \geq 1$ and that $K_S$ has a complex prime.
Let $U$ be an open normal subgroup of $G_{K,S}$.
Then the conjugation action of $G_{K,S}$ on $U$ is faithful.
%%the canonical homomorphism $G_{K,S} \to \Aut(U)$ defined by the conjugate action is injective.
In particular, $G_{K,S}$ has a trivial center.
\end{cor}
\begin{proof}
Let $\tau \in G_{K,S}$ such that the automorphism of $U$ induced by the conjugation action of $\tau$ is trivial.
Then, by Proposition \ref{0.2}, $\tau$ is trivial.
\end{proof}

\section{Isomorphisms of fields}

In this section, we 
develop a way, 
%%whose arguements are 
based on \cite{Shimizu}, \S 3, 
to show that 
isomorphisms of Galois groups are induced by 
(unique) field isomorphisms. 
%%using the results obtained so far (Theorem \ref{3.4}).
%%under some assumptions.
%%assuming the good local correspondence holds.
%%The arguements are essentially based on \cite{Shimizu}, \S 3, and sharpening
By using this, we prove the main result.

In the rest of this 
%%section, 
paper, 
fix an algebraic closure $\overline{\bQ}$ of $\bQ$, 
and 
suppose that 
all number fields and all algebraic extensions of them 
are subfields of $\overline{\bQ}$.
%%For a number field $K$, we 
Write $\widetilde{K}$ for the Galois closure of $K/\bQ$.
For $i=1,2$, let $K_i$ be a number field and $S_i$ a set of primes of $K_i$ with $P_{K_i,\infty} \subset S_i$.

\begin{prop}\label{2.2}
%%For $i=1,2$, let $K_i$ be a number field, $S_i$ a set of primes of $K_i$ with $P_{K_i,\infty} \subset S_i$, 
Let $\sigma :G_{K_1,S_1}\isom G_{K_2,S_2}$ be an isomorphism, and 
for $i=1,2$, 
$U_i$ an open normal subgroup of $G_{K_i,S_i}$ with $\sigma(U_1) = U_2$.
For $i=1,2$, write $L_i$ for the finite Galois subextension of $K_{i,{S_i}}/K_i$ corresponding to $U_i$.
Let $T_i \subset S_{i,f}(L_i)$ for $i=1,2$.
Assume that the following conditions hold:
\begin{itemize}

\item[(a)]
%%There exists $T_i \subset S_{i,f}(L_i)$ for $i=1,2$ such that the 
The 
good local correspondence between $T_1$ and $T_2$
holds for $\sigma|_{U_1}$ (see \cite[Definition 2.3]{Shimizu}).

\item[(b)]
There exists a finite extension $L/L_1L_2$ such that 
$L/\bQ$ is Galois and 
$\delta(T_i(L)) \neq 0$ for one $i$.

\item[(c)]
%%There exists a prime number $l \in P_{S_1,f} \cap P_{S_2,f}$.
$P_{S_1,f} \cap P_{S_2,f} \neq \emptyset$.

\item[(d)]
$L_i$ has a complex prime for one $i$.

\end{itemize}
Then 
there exists $\tau \in G(L/\bQ)$ such that 
for any 
open normal subgroups $U_1'$, $U_2'$ of $G_{K_1,S_1}$,  $G_{K_2,S_2}$ containing $U_1$, $U_2$, respectively, with $\sigma(U_1') = U_2'$, 
it follows that 
$K_1=\tau(K_2)$, 
$L_1'=\tau(L_2')$ and 
the isomorphism: $G(L_1'/K_1) \isom G(L_2'/K_2)$ induced by $\sigma$ 
coincides with the isomorphism induced by $\tau|_{L_2'}$, 
where $L_1'$, $L_2'$ are finite Galois subextensions of $L_1/K_1$, $L_2/K_2$, corresponding to $U_1'$, $U_2'$, respectively.
\end{prop}
\begin{proof}
By symmetry, we may assume that $L_2$ has a complex prime.
Write $\overline{\sigma}_{L_1}:G(L_1/K_1) \isom G(L_2/K_2)$ for the isomorphism induced by $\sigma$.
Take 
%%a prime number 
$l \in P_{S_1,f} \cap P_{S_2,f}$.
Write $\widetilde{\sigma}_{L_1, l}:\Gamma_{L_1, l} \isom \Gamma_{L_2, l}$ for the isomorphism induced by $\sigma|_{U_1}$. 
%%and for any $\tau \in G(L/\bQ)$, $\widetilde{(\tau|_{L_2})}^\ast_{l}:\Gamma_{\tau(L_2), l} \isom \Gamma_{L_2, l}$ for the isomorphism induced by $\tau|_{L_2}$.
%%
By \cite[Proposition 3.3]{Shimizu}, 
there exists $\tau \in G(L/\bQ)$ such that 
%%$K_1^{(\infty,l)}\tau(K_2) = K_1\tau(K_2)^{(\infty,l)}$.
$\widetilde{\sigma}_{L_1, l} \circ \pi_{L/L_1, l} = \widetilde{(\tau|_{L_2})}^\ast_{l} \circ \pi_{L/\tau(L_2), l}$, 
where $\widetilde{(\tau|_{L_2})}^\ast_{l}:\Gamma_{\tau(L_2), l} \isom \Gamma_{L_2, l}$ is the isomorphism induced by $\tau|_{L_2}$.
Then we have $\widetilde{(\tau|_{L_2})}_{l}^{\ast -1} \circ \widetilde{\sigma}_{L_1, l} \circ \pi_{L/L_1, l} = \pi_{L/\tau(L_2), l}$.
By \cite[Lemma 3.4]{Shimizu}, we obtain 
$L_1^{(\infty,l)}\tau(L_2) = L_1\tau(L_2)^{(\infty,l)}$, 
and hence 
$L_1^{(\infty,l)}L = L\tau(L_2)^{(\infty,l)}$.
By \cite[Proposition 3.5]{Shimizu}, we obtain 
$L_1=\tau(L_2)$.
Further, since $\Image(\pi_{L/L_1, l})$ is open in $\Gamma_{L_1, l}$, 
$\widetilde{\sigma}_{L_1, l}$ coincides with $\widetilde{(\tau|_{L_2})}^\ast_{l}$.

Write $(\tau|_{L_2})^\ast: \Aut(L_1) \isom \Aut(L_2)$ for the isomorphism induced by $\tau|_{L_2}$.
By the 
%%above 
equality: $\widetilde{\sigma}_{L_1, l} = \widetilde{(\tau|_{L_2})}^\ast_{l}$, 
for any $\tau_1 \in G(L_1/K_1)$, 
%%the elements in $\Aut(\Gamma_{L_2, l})$
the actions of $\overline{\sigma}_{L_1}(\tau_1)$ and $(\tau|_{L_2})^\ast(\tau_1)$ on $\Gamma_{L_2, l}$ coincide.
Therefore, by Lemma \ref{2.1}, we obtain $\overline{\sigma}_{L_1}(\tau_1) = (\tau|_{L_2})^\ast(\tau_1)$, 
%% for any $\tau_1 \in G(L_1/K_1)$, 
so that $\overline{\sigma}_{L_1} = (\tau|_{L_2})^\ast$.
Hence 
%%for any $L_1'$, $L_2'$ as in the assertion, 
$\tau|_{L_2}$ is compatible with the actions of 
$G(L_2/K_2)$ and $G(L_1/K_1)$ 
%%$G(L_2/L_2')$ and $G(L_1/L_1')$ 
on $L_2$ and $L_1$.
Thus, 
for any $U_1'$, $U_2'$ 
%%$L_1'$, $L_2'$ 
as in the assertion, 
%%Since $\tau$ and $(\tau|_{L_2})^\ast$
%%$\tau|_{L_2}$ induces an isomorphism between $L_2^{G(L_2/K_2)} = K_2$ and $L_1^{G(L_1/K_1)} = K_1$.
$\tau|_{L_2}$ induces an isomorphism between $L_2^{G(L_2/L_2')} = L_2'$ and $L_1^{G(L_1/L_1')} = L_1'$.
Further, the isomorphism: $G(L_1'/K_1) \isom G(L_2'/K_2)$ induced by $\sigma$ 
coincides with the isomorphism induced by $\tau|_{L_2'}$.
\end{proof}

\begin{theorem}\label{2.3}
%%For $i=1,2$, let $K_i$ be a number field, $S_i$ a set of primes of $K_i$ with $P_{K_i,\infty} \subset S_i$, 
Let $\sigma :G_{K_1,S_1}\isom G_{K_2,S_2}$ be an isomorphism, and 
for $i=1,2$, 
$V_i$ a closed normal subgroup of $G_{K_i,S_i}$ with $\sigma(V_1) = V_2$.
For $i=1,2$, write $M_i$ for the Galois subextension of $K_{i,{S_i}}/K_i$ corresponding to $V_i$.
Assume that the following conditions hold:
\begin{itemize}

\item[(a)]
$\# P_{S_i,f} \geq 2$ for $i=1,2$. 
%%ローコレ用

\item[(b)]
For one $i$ and for any finite Galois subextension $L_i$ of $M_i/K_i$, 
$\delta(P_{S_i,f} \cap \cs(L_i/\bQ)) \neq 0$.

\item[(c)]
%%For the other $i$, 
For the $i$ in condition (b), 
there exists 
a prime number $l \in P_{S_1,f} \cap P_{S_2,f}$ 
such that 
%%$S_{i}$ 
$S_{3-i}$ 
satisfies condition $(\star_l)$ (see \cite[Definition 1.16]{Shimizu}).
%%ローコレ用

\item[(d)]
$M_i$ has a complex prime for one $i$.

\end{itemize}
Then 
there exists 
%%$\tau \in G_\bQ$ 
%%$\tau \in G(\widetilde{M_1M_2}/\bQ)$ 
an isomorphism $\tau:M_2 \isom M_1$ 
such that 
$K_1=\tau(K_2)$
%%, $M_1=\tau(M_2)$ 
and 
the isomorphism: $G(M_1/K_1) \isom G(M_2/K_2)$ induced by $\sigma$ 
coincides with the isomorphism induced by $\tau$.
\end{theorem}

\begin{proof}
%%[Proof of Theorem \ref{2.3}]
%%We use the same notations as in Prop
%%For $i=1,2$, let $U_i$ be any open normal subgroup of $G_{K_i,S_i}$ 
%%, corresponding to the finite Galois subextension $L_i$ of $M_i/K_i$, 
%%such that $V_i \subset U_i$, 
%%containing $V_i$ with $\sigma(U_1) = U_2$.
%%, and $L_i$ has a complex prime for one $i$.
Let $U_1$ be any open normal subgroup of $G_{K_1,S_1}$ containing $V_1$. 
Set $U_2 \defeq \sigma(U_1)$.
For $i=1,2$, write $L_i$ for the finite Galois subextension of $K_{i,{S_i}}/K_i$ corresponding to $U_i$.
%%Let $L_1'$, $L_2'$ be any finite Galois subextensions of $L_1/K_1$, $L_2/K_2$, respectively, corresponding to open normal subgroups $U_1'$, $U_2'$ of $G_{K_1,S_1}$,  $G_{K_2,S_2}$ containing $U_1$, $U_2$, respectively, with $\sigma(U_1') = U_2'$.
Let $U_1'$, $U_2'$ be any open normal subgroups of $G_{K_1,S_1}$,  $G_{K_2,S_2}$ containing $U_1$, $U_2$, respectively, with $\sigma(U_1') = U_2'$.
For $i=1,2$, write $L_i'$ for the finite Galois subextension of $L_i/K_i$ corresponding to $U_i'$.
%%, corresponding to finite Galois subextensions $L_1'$, $L_2'$ of $L_1/K_1$, $L_2/K_2$, respectively.
We set 
\begin{equation*}
\mathfrak{A}_{U_1}\defeq 
\left\{ \tau \in G_\bQ \left|
\begin{array}{l}
%%\text{$K_1=\tau(K_2)$, $L_1=\tau(L_2)$ and the isomorphism: $G(L_1/K_1) \isom G(L_2/K_2)$}\\
%%\text{induced by $\sigma$ coincides with the isomorphism induced by $\tau|_{L_2}$}
\text{
for any $U_1'$, $U_2'$
%%$L_1'$, $L_2'$ 
as above, 
$K_1=\tau(K_2)$, $L_1'=\tau(L_2')$ and the isomorphism:}\\ 
\text{$G(L_1'/K_1) \isom G(L_2'/K_2)$ induced by $\sigma$ 
coincides with the isomorphism}\\ 
\text{induced by $\tau|_{L_2'}$
}
\end{array}
\right.\right\}.
\end{equation*}
%%where $\overline{\sigma}_{L_1},(\tau|_{L_2})^\ast:G(L_1/K_1) \isom G(L_2/K_2)$ for the isomorphism induced by $\sigma$, $\tau|_{L_2}$, respectively.
%%$(\tau|_{L_2})^\ast: \Aut(L_1) \isom \Aut(L_2)$ for the isomorphism induced by $\tau|_{L_2}$.
Note that $\mathfrak{A}_{U_1}$ is a closed subset of $G_\bQ$.
In order to prove the existence of $\tau$ in the assertion, 
it suffices to show that 
$\mathfrak{A}_{U_1} \neq \emptyset$ 
for every $U_1$ 
%%and $U_2$ 
as above.
Indeed, having shown this, we obtain 
$\cap_{U_1} \mathfrak{A}_{U_1} \neq \emptyset$. 
%%\varprojlim_{U_1}
%%since $\mathfrak{A}_{U_1}$ is non-empty closed subset of the compact set $G_\bQ$.
%%
Moreover, we 
%%We 
may assume that $L_i$ has a complex prime for one $i$.

By symmetry, we may assume that the $i$ in condition (b) is $1$.
%%Take a prime number $l \in P_{S_1,f} \cap P_{S_2,f}$ such that $S_{2}$ satisfies condition $(\star_l)$.
By \cite[Lemma 4.1]{Shimizu}, 
we have 
$\delta_{\sup}(S_1) \geq \delta_{\sup}((P_{S_1,f} \cap \cs(K_1/\bQ))(K_1)) 
= [K_1 : \bQ]\delta_{\sup}(P_{S_1,f} \cap \cs(K_1/\bQ)) 
> 0$.
Hence, by \cite[Proposition 1.21]{Shimizu}, 
$S_1$ satisfies condition $(\star_l)$ 
for the prime number $l$ in condition (c).
Therefore, by \cite[Theorem 2.6]{Shimizu}, 
the local correspondence between $S_{1,f}$ and $S_{2,f}$
holds for $\sigma$.
By \cite[Proposition 2.8]{Shimizu}, 
%%$P_{S_1,f} \cap \cs(K_1/\bQ) =P_{S_2,f} \cap \cs(K_2/\bQ)$, 
$P_{S_1,f} = P_{S_2,f}$ and 
the good local correspondence between $P_{S_1,f}(K_1)$ and $P_{S_2,f}(K_2)$ holds for $\sigma$. 
Further, by \cite[Remark 2.7]{Shimizu}, 
the good local correspondence between $P_{S_1,f}(L_1)$ and $P_{S_2,f}(L_2)$ holds for $\sigma|_{U_1}$. 
Again by \cite[Proposition 2.8]{Shimizu}, we have 
$P_{S_1,f} \cap \cs(L_1/\bQ) =P_{S_2,f} \cap \cs(L_2/\bQ)$. 
Now, we set 
$L \defeq \widetilde{L_1 L_2}$.
Then we have 
\begin{equation*}
\begin{split}
P_{S_1,f} \cap \cs(L/\bQ) 
&= P_{S_1,f} \cap \cs(L_1 L_2/\bQ) \\
&= P_{S_1,f} \cap \cs(L_1/\bQ) \cap \cs(L_2/\bQ)\\
&= P_{S_1,f} \cap \cs(L_1/\bQ).
\end{split}
\end{equation*}
%%so that $\delta_{\sup}(P_{S_1,f} \cap \cs(L/\bQ)) > 0$.
Therefore, by \cite[Lemma 4.1]{Shimizu}, 
$$\delta_{\sup}(P_{S_1,f}(L)) 
%%= \delta_{\sup}((P_{S_1,f} \cap \cs(M/\bQ))(M)) 
= [L : \bQ]\delta_{\sup}(P_{S_1,f} \cap \cs(L/\bQ)) 
= [L : \bQ]\delta_{\sup}(P_{S_1,f} \cap \cs(L_1/\bQ)) 
> 0.$$
Thus, by Proposition \ref{2.2}, we obtain 
$\mathfrak{A}_{U_1} \neq \emptyset$. 
\end{proof}

\begin{rem}\label{2.4}
By \cite[Lemma 4.1]{Shimizu}, 
condition (b) in Theorem \ref{2.3} is equivalent to the condition: ``for one $i$ and for any finite Galois subextension $L_i$ of $M_i/K_i$, 
$\delta(P_{S_i,f}(\widetilde{L_i})) \neq 0$".
Further, again by \cite[Lemma 4.1]{Shimizu}, 
this condition holds if the condition: ``for one $i$ and for some finite extension $L_i$ of $K_i$ (not necessary contained in $K_{i,{S_i}}$), 
$\delta_{\sup}(P_{S_i,f}(L_i)) = 1$" is satisfied.
%%Note that 
In particular, 
if $\delta_{\sup}(S_i) = 1$ for one $i$, 
then $\delta_{\sup}(S_i(\widetilde{K_i})) = 1$ by \cite[Lemma 4.1]{Shimizu}, 
and hence $\delta_{\sup}(P_{S_i,f}(\widetilde{K_i})) = 1$ by applying  \cite[Lemma 4.5]{Shimizu} to $K=\widetilde{K_i}$, 
so that condition (b) holds.
\end{rem}

The main result in this paper is the following.

\begin{theorem}\label{2.5}
%%For $i=1,2$, let $K_i$ be a number field, $S_i$ a set of primes of $K_i$ with $P_{K_i,\infty} \subset S_i$, 
Assume that the following conditions hold:
\begin{itemize}

\item[(a)]
$\# P_{S_i,f} \geq 2$ for $i=1,2$. 
%%ローコレ用

\item[(b)]
For one $i$ and for any finite Galois subextension $L_i$ of $K_{i,{S_i}}/K_i$, 
$\delta(P_{S_i,f} \cap \cs(L_i/\bQ)) \neq 0$.

\item[(c)]
%%For the other $i$, 
For the $i$ in condition (b), 
there exists 
a prime number $l \in P_{S_1,f} \cap P_{S_2,f}$ 
such that 
%%$S_{i}$ 
$S_{3-i}$ 
satisfies condition $(\star_l)$ (see \cite[Definition 1.16]{Shimizu}).
%%ローコレ用

%%\item[(d)]$M_i$ has a complex prime for one $i$.

\end{itemize}
Let $\sigma :G_{K_1,S_1}\isom G_{K_2,S_2}$ be an isomorphism.
Then 
there exists 
%%$\tau \in G_\bQ$ 
a unique isomorphism $\tau:K_{2,{S_2}} \isom K_{1,{S_1}}$ 
such that 
$K_1=\tau(K_2)$
%%, $K_{1,{S_1}}=\tau(K_{2,{S_2}})$ 
and 
%%the isomorphism: $G(K_{1,{S_1}}/K_1) \isom G(K_{2,{S_2}}/K_2)$ induced by $\sigma$ 
$\sigma$ 
coincides with the isomorphism induced by $\tau$.
In other words, the canonical map: $\Iso(K_{2,{S_2}}/K_2, K_{1,{S_1}}/K_1) \to \Iso(G_{K_1,S_1}, G_{K_2,S_2})$ is bijective.
\end{theorem}

\begin{proof}
For $i=1,2$ and $l \in P_{S_i,f}$, 
we have $\mu_{l^2} \subset K_{i,{S_i}}$, so that 
$K_{i,{S_i}}$ is totally imaginary.
Hence the existence (resp. the uniqueness) of $\tau$ in the assertion follows from Theorem \ref{2.3} (resp. Proposition \ref{0.2}).
\end{proof}

\section{Corollaries}
In this section, we see some applications of the main theorem.

\begin{lem}\label{3.1}
%%For $i=1,2$, let $K_i$ be a number field and $S_i$ a set of primes of $K_i$ with $P_{K_i,\infty} \subset S_i$. Then the 
The 
canonical maps: 
\begin{equation*}
\begin{split}
&\Iso(\cO_{{K_2},S_2}, \cO_{{K_1},S_1}) \to \Iso((K_2,S_2), (K_1,S_1)),\\
&\Iso(\cO_{K_{2,{S_2}},S_2(K_{2,{S_2}})}/\cO_{{K_2},S_2}, \cO_{K_{1,{S_1}},S_1(K_{1,{S_1}})}/\cO_{{K_1},S_1}) \to \Iso((K_{2,{S_2}}/K_2, S_2), (K_{1,{S_1}}/K_1, S_1))
\end{split}
\end{equation*}
are bijective.
\end{lem}
\begin{proof}
The inverse maps are induced by restriction.
\end{proof}

\begin{lem}\label{3.2}
Assume 
$P_{\infty} \subset S$ and that 
$\# P_{S,f} \geq 1$.
Then 
all finite primes in $S_f$ and all real primes in $P_{\infty}$ are ramified in $K_S/K$.
\end{lem}
\begin{proof}
Take $l \in P_{S,f}$.
Then all finite primes in $P_l$ and all real primes in $P_{\infty}$ are ramified in $K(\mu_{l^\infty})/K$.
Further, by the proof of \cite[Lemma 2.3]{Ivanov}, all finite primes in $S_f \setminus P_l$ are ramified in $K_S/K$.
\end{proof}

\begin{lem}\label{3.3}
%%For $i=1,2$, let $K_i$ be a number field and $S_i$ a set of primes of $K_i$ with $P_{K_i,\infty} \subset S_i$.
For one $i$, assume 
%%$P_{K_i,\infty} \subset S_i$ and that 
$\# P_{S_i,f} \geq 1$.
Then 
the canonical inclusion: $$\Iso((K_{2,{S_2}}/K_2, S_2), (K_{1,{S_1}}/K_1, S_1)) \hookrightarrow \Iso(K_{2,{S_2}}/K_2, K_{1,{S_1}}/K_1)$$ is bijective.
\end{lem}
\begin{proof}
By symmetry, we may assume that $\# P_{S_1,f} \geq 1$.
Take $l \in P_{S_1,f}$.
Then $\mu_{l^\infty} \subset K_{1,{S_1}}$.
If $\Iso(K_{2,{S_2}}/K_2, K_{1,{S_1}}/K_1)$ is not empty, then $\mu_{l^\infty} \subset K_{2,{S_2}}$, so that $l \in P_{S_2,f}$.
Thus, the assertion 
%%immediately 
follows from Lemma \ref{3.2}.
\end{proof}

%%For $i=1,2$, let $K_i$ be a number field and $S_i$ a set of primes of $K_i$.
An element in $\Iso((K_2,S_2), (K_1,S_1))$ 
can be extended to an element in 
\\
$\Iso((K_{2,{S_2}}/K_2, S_2), (K_{1,{S_1}}/K_1, S_1)) \subset \Iso(K_{2,{S_2}}/K_2, K_{1,{S_1}}/K_1)$, and therefore induces a well-defined element in $\OutIso(G_{K_1,S_1}, G_{K_2,S_2})$.
As a corollary of Theorem \ref{2.5}, 
we obtain the following, which is a generalization of {\cite[(12.2.2) Corollary]{NSW}}.

\begin{cor}\label{3.4}
Notations and assumptions are the same as in Theorem \ref{2.5}.
Then 
the canonical map: 
$\Iso((K_2,S_2), (K_1,S_1)) \to \OutIso(G_{K_1,S_1}, G_{K_2,S_2})$
is bijective.
\end{cor}
\begin{proof}
$G_{K_2,S_2}$ acts on $\Iso(K_{2,{S_2}}/K_2, K_{1,{S_1}}/K_1)$ by the rule $\sigma(\phi) \defeq \phi\circ\sigma^{-1}$, 
and by Lemma \ref{3.3}, 
we have a bijection: 
%%an isomorphism: 
$\Iso(K_{2,{S_2}}/K_2, K_{1,{S_1}}/K_1)/G_{K_2,S_2} \simeq \Iso((K_2,S_2), (K_1,S_1))$.
By Theorem \ref{2.5}, we have a bijection: 
%%an isomorphism: 
$\Iso(K_{2,{S_2}}/K_2, K_{1,{S_1}}/K_1) \to \Iso(G_{K_1,S_1}, G_{K_2,S_2})$, which is easily seen to be $G_{K_2,S_2}$-invariant if we let $G_{K_2,S_2}$ act by inner automorphisms on the right-hand side (cf. Notations).
Thus, factoring out by the $G_{K_2,S_2}$-actions,
we obtain the required bijection.
%%isomorphism.
%%this isomorphism induces the isomorphism in the assertion.
\end{proof}

\begin{cor}\label{3.5}

Assume $P_{\infty} \subset S$ 
and that 
for any finite Galois subextension $L$ of $K_{S}/K$, 
$\delta(P_{S,f} \cap \cs(L/\bQ)) \neq 0$.
Then 
there is a canonical isomorphism: 
$\Aut(K,S) \isom \Out(G_{K,S})$.
%%where $\Aut(K,S) \defeq \Iso((K,S), (K,S))$ and $\Out(G_{K,S}) \defeq \OutIso(G_{K,S}, G_{K,S})$.
%%Moreover,
%%In particular, 
%%when $K=\bQ$, the canonical homomorphism: $G_{\bQ,S} \to \Aut(G_{\bQ,S})$ induced by the conjugation action is bijective.
\end{cor}
\begin{proof}
By \cite[Lemma 4.1]{Shimizu} and \cite[Proposition 1.21]{Shimizu}, 
$S$ satisfies condition $(\star_l)$ for any $l \in P_{S,f}$.
Therefore, the assertion follows immediately from Corollary \ref{3.4}.
%%Therefore, by Corollary \ref{3.4}, we obtain the first isomorphism.
%%Assume $K=\bQ$. Then $\Aut(\bQ,S)$ is trivial, so that $\Out(G_{\bQ,S})$ is also trivial. Therefore, $\Inn(G_{\bQ,S})=\Aut(G_{\bQ,S})$. By Proposition \ref{1.2}, the canonical homomorphism: $G_{\bQ,S} \to \Inn(G_{\bQ,S})$ is bijective.
\end{proof}

The following is a generalization of {\cite[(12.2.3) Corollary]{NSW}}.

\begin{cor}\label{3.6}
Notations and assumptions are the same as in Corollary \ref{3.5}.
Further, assume $\Aut(K,S)$ is trivial.
Then the canonical homomorphism: 
$G_{K,S} \to \Aut(G_{K,S})$ induced by the conjugation action is bijective.
In particular, all automorphisms of $G_{K,S}$ are inner.
\end{cor}
\begin{proof}
By Corollary \ref{3.5}, $\Out(G_{K,S})$ is trivial.
Therefore, $\Inn(G_{K,S})=\Aut(G_{K,S})$.
By Proposition \ref{0.2}, the canonical homomorphism: 
$G_{K,S} \to \Inn(G_{K,S})$ is bijective.
\end{proof}

\begin{rem}\label{3.7}
Assume $P_{\infty} \subset S$.
Write $\pi_1(\cO_{{K},S})$ for the etale fundamental group of $\cO_{{K},S}$.
Then there exists a canonical isomorphism: $\pi_1(\cO_{{K},S}) \simeq G_{K,S}$.

By the canonical isomorphisms in Lemma \ref{3.1} and Lemma \ref{3.3}, 
%%obtained in this section so far, 
we can identify the canonical map: $\Iso(K_{2,{S_2}}/K_2, K_{1,{S_1}}/K_1) \to \Iso(G_{K_1,S_1}, G_{K_2,S_2})$ in the assertion of Theorem \ref{2.5} 
with the canonical map: $$\Iso(\cO_{K_{2,{S_2}},S_2(K_{2,{S_2}})}/\cO_{{K_2},S_2}, \cO_{K_{1,{S_1}},S_1(K_{1,{S_1}})}/\cO_{{K_1},S_1}) \to \Iso(\pi_1(\cO_{{K_1},S_1}), \pi_1(\cO_{{K_2},S_2})).$$
Similarly, we can replace the fields and the Galois groups in the corollaries in this section 
%%paper 
by the rings and the etale fundamental groups, respectively.
\end{rem}

\section{Appendix: Intersections of decomposition groups}
In this section, we study intersections of decomposition groups in $G_{K,S}$. 
%%at primes in $S_f(K_S)$.
The results not only are interesting in themselves, 
but also give a 
%%another 
proof of an a little weaker version of Proposition \ref{0.2}.
%%Not only are these results interesting in themselves, but they also give another proof of the theorem.

\begin{lem}\label{1.1}
Assume that 
$P_{\infty} \subset S$ and that 
$\# P_{S,f} \geq 2$.
Let $\overline{\p}, \overline{\q} \in S_f(K_S)$.
%%and $D_{\overline{\p}}$, $D_{\overline{\q}}$ dec subgroups, respectively.
Then 
$D_{\overline{\p}}(K_S/K) = D_{\overline{\q}}(K_S/K)$
if and only if $\overline{\p} = \overline{\q}$.
\end{lem}
\begin{proof}
The assertion follows immediately from \cite[Corollary 2.7(ii)]{Ivanov}.
\end{proof}

\begin{lem}\label{1.2}
Let $p$ be a prime number, $\kappa$ a $p$-adic field and $N$ a non-trivial closed normal subgroup of $G_{\kappa}$. 
Then 
%%the ramification subgroup of $N$ 
$V_{\kappa} \cap N$ 
is a topologically infinitely generated pro-$p$ subgroup.
%%has a topologically infinitely generated pro-$p$ subgroup.
\end{lem}
\begin{proof}
The assertion follows immediately from \cite[(1.4) Satz]{Pop}.
\end{proof}

\begin{lem}\label{1.3}
Let $p,l$ be distinct prime numbers, 
%%$\kappa$ a $p$-adic field 
$\kappa/\bQ_p$ an (a possibly infinite) algebraic extension 
and $\lambda/\kappa$ a Galois extension. 
%%and $N$ a closed subgroup of $G(\lambda/\kappa)$. 
Then $G(\lambda/\kappa)$ does not have a topologically infinitely generated pro-$l$ subgroup.
\end{lem}
\begin{proof}
We may assume 
that $\kappa$ is a $p$-adic field and 
that $\lambda = \overline{\kappa}$.
%%Let $H$ be any pro-$l$ subgroup of $G_{\kappa}$, and $G_{\kappa,l}$ an $l$-Sylow subgroup of $G_{\kappa}$ containing $H$. Note that $G_{\kappa,l}$ is of $l$-decomposition type (cf. \cite[Definition 2.1 and 2.2, Local situation]{Ivanov}). By \cite[Lemma 2.2]{Ivanov}, $G_{\kappa,l}$ does not have a topologically infinitely generated pro-$l$ subgroup.
Let $G_{\kappa,l}$ be any $l$-Sylow subgroup of $G_{\kappa}$. 
Then there exists an exact sequence: 
$1\to \bZ_l \to G_{\kappa,l} \to \bZ_l \to 1$, 
and hence all subgroups of $G_{\kappa,l}$ are topologically generated by at most two elements (cf. \cite[2.2, Local situation and Lemma 2.2]{Ivanov}).
\end{proof}

\begin{prop}\label{1.4}
%%Assume that $P_{\infty} \subset S$. 
Let $\overline{\p} \in {S_f}(K_S)$ and $\q \in P_K \setminus \{ \overline{\p}|_K \}$.
Assume that 
$\overline{\p}|_{\bQ} \neq {\q}|_{\bQ}$ if $\q \in S_f$, 
%%or $D_{\overline{\p}}(K_S/K) \cap I_{\overline{\q}}(K_S/K) =1$
and that $D_{\overline{\p}}(K_S/K)$ is full.
Then 
$D_{\overline{\p}}(K_S/K) \cap (\cap_{\overline{\q} \in \{ \q \}(K_S)} D_{\overline{\q}}(K_S/K))$ is trivial.
\end{prop}
\begin{proof}
Let $p=\overline{\p}|_{\bQ}$ and 
$N = D_{\overline{\p}}(K_S/K) \cap (\cap_{\overline{\q} \in \{ \q \}(K_S)} D_{\overline{\q}}(K_S/K))$.
%%Write $N$ for the closed normal subgroup $D_{\overline{\p}}(K_S/K) \cap (\cap_{\overline{\q} \in \{ \q \}(K_S)} D_{\overline{\q}}(K_S/K))$ of $D_{\overline{\p}}(K_S/K)$.
Assume that $N$ is non-trivial.
By the fullness of $D_{\overline{\p}}(K_S/K)$ 
and Lemma \ref{1.2}, $N$ has a topologically infinitely generated pro-$p$ subgroup.
%%If $\q \notin S_f$, then $\q$ is unramified in $K_S/K$ or an archimedean prime, and hence for any $\overline{\q} \in \{ \q \}(K_S)$, $D_{\overline{\q}}(K_S/K)$ is pro-cyclic, a contradiction.
Assume $\q \notin S_f$. Then $\q$ is unramified in $K_S/K$ or an archimedean prime. Therefore, for any $\overline{\q} \in \{ \q \}(K_S)$, $D_{\overline{\q}}(K_S/K)$ is pro-cyclic, so that all subgroups of $D_{\overline{\q}}(K_S/K)$ are also pro-cyclic, a contradiction.
Assume $\q \in S_f$. Then $p \neq {\q}|_{\bQ}$ by assumption.
By Lemma \ref{1.3}, 
for any $\overline{\q} \in \{ \q \}(K_S)$, $D_{\overline{\q}}(K_S/K)$ 
does not have a topologically infinitely generated pro-$p$ subgroup, a contradiction.
\end{proof}

\begin{cor}\label{1.5}
Assume that 
$P_{\infty} \subset S$ and that 
$\# P_{S,f} \geq 2$.
Then $\cap_{\overline{\p} \in S_f(K_S)} D_{\overline{\p}}(K_S/K)$ is trivial.
\end{cor}
\begin{proof}
Take $\overline{\p} \in P_{S,f}(K_S)$ and 
%%$\q \in P_{S,f}(K)$ with $\overline{\p}|_{\bQ} \neq \q|_{\bQ}$.
$\q \in S_f$ with $\overline{\p}|_{\bQ} \neq \q|_{\bQ}$.
%%$\q \in S_f \setminus P_{\overline{\p}|_{\bQ}}$.
By \cite[Th\'{e}or\`{e}me 5.1]{Chenevier-Clozel}, $D_{\overline{\p}}(K_S/K)$ 
%%and $D_{\overline{\q}}(K_S/K)$ are 
is 
full (see \cite[Corollary 2.6]{Ivanov}).
By Proposition \ref{1.4}, we have $D_{\overline{\p}}(K_S/K) \cap (\cap_{\overline{\q} \in \{ \q \}(K_S)} D_{\overline{\q}}(K_S/K))=1$.
Thus, we obtain $\cap_{\overline{\p} \in S_f(K_S)} D_{\overline{\p}}(K_S/K)=1$.
%%It follows from Lemma \ref{1.2} that $\cap_{\overline{\p} \in S_f(K_S)} D_{\overline{\p}}(K_S/K)$ is trivial.
\end{proof}

The following is weaker than Proposition \ref{0.2}.
However, in the proof, we do not use Lemma \ref{2.1}.
%%use only results in this section.

\begin{prop}\label{1.7}
Assume that 
$P_{\infty} \subset S$ and that 
$\# P_{S,f} \geq 2$.
Let $\tau \in \Aut(K_S)$. 
Assume $\tau(K)=K$ and that the automorphism of $G_{K,S}$ induced by the conjugation action of $\tau$ 
%%by $\tau$ by conjugation 
is trivial.
Then $\tau$ is trivial.
\end{prop}
\begin{proof}
%%Write $\tau^\ast$ for the automorphism of $G_{K,S}$ induced by $\tau$. $G_{K,S} \subset \Aut(K_S)$.
Write $K_0$ for the $\Aut(K_S)$-invariant subfield of $K_S$. 
Then $K_S/K_0$ is Galois.
Let $\overline{\p} \in S_f(K_S)$. Then 
we have 
$D_{\tau\overline{\p}}(K_S/K) 
= \tau^{-1} D_{\overline{\p}}(K_S/K) \tau
= D_{\overline{\p}}(K_S/K)$ in $\Aut(K_S)$.
%%右作用
By Lemma \ref{3.3}, $\tau\overline{\p} \in S_f(K_S)$.
Therefore, by Lemma \ref{1.1}, we obtain $\tau\overline{\p}=\overline{\p}$, 
and hence $\tau \in D_{\overline{\p}}(K_S/K_0)$.
Thus, $\tau \in N \defeq \cap_{\overline{\p} \in S_f(K_S)} D_{\overline{\p}}(K_S/K_0)$.

By Corollary \ref{1.5}, $N \cap G_{K,S} = \cap_{\overline{\p} \in S_f(K_S)} D_{\overline{\p}}(K_S/K)$ is trivial. Hence $N$ is finite.
As in the proof of Corollary \ref{1.5}, for $\overline{\p} \in P_{S,f}(K_S)$, $D_{\overline{\p}}(K_S/K)$ is full, and hence $D_{\overline{\p}}(K_S/K_0)$ is also full, 
so that $D_{\overline{\p}}(K_S/K_0)$ is torsion-free by \cite[(7.1.8) Theorem (i)]{NSW}.
Therefore, $N$ is trivial.
Thus, $\tau$ is trivial.
\end{proof}

%%\noindent
%%Ryoji Shimizu\\
%%Research Institute for Mathematical Sciences\\
%%Kyoto University\\
%%KYOTO 606-8502\\
%%Japan\\
%%shimizur@kurims.kyoto-u.ac.jp\\


\begin{thebibliography}{Y}

\bibitem[Chenevier-Clozel]{Chenevier-Clozel} Chenevier, G., Clozel, L., Corps de nombres peu ramifi\'{e}s et formes automorphes autoduales, J. of the AMS, vol. 22 (2009), no. 2, 467--519.

\bibitem[Ivanov]{Ivanov2} Ivanov, A., Arithmetic and anabelian theorems for stable sets in number fields, Dissertation, Universit\"at Heidelberg, 2013.

\bibitem[Ivanov2]{Ivanov} Ivanov, A., On some anabelian properties of arithmetic curves, Manuscripta Mathematica 144 (2014), no. 3, 545--564.

\bibitem[Ivanov3]{Ivanov3} Ivanov, A., On a generalization of the Neukirch-Uchida theorem, Moscow Mathematical J. 17 (2017), no. 3, 371--383.


%%\bibitem[Mochizuki]{Mochizuki} Mochizuki, S., Topics in Absolute Anabelian Geometry I: Generalities, \textit{J. Math. Sci. Univ. Tokyo} \textbf{19} (2012), no. 2, 139--242.


\bibitem[Neukirch]{Neukirch} Neukirch, J., Kennzeichnung der $p$-adischen und der endlichen algebraischen Zahlkörper, Invent. Math. 6 (1969), 296--314.

%%\bibitem[Neukirch2]{Neukirch2} Neukirch, J., Kennzeichnung der endlich-algebraischen Zahlkörper durch die Galoisgruppe der maximal auflösbaren Erweiterungen, J. Reine Angew. Math. 238 (1969), 135--147.

%%\bibitem[Neukirch3]{Neukirch3} Neukirch, J., Algebraic Number Theory, Grundlehren der Mathematischen Wissenschaften, 322. Springer-Verlag, Berlin, 1999.

\bibitem[NSW]{NSW}
Neukirch, J., Schmidt, A. and Wingberg, K., 
Cohomology of number fields, Second edition, Grundlehren der Mathematischen Wissenschaften, 323. Springer-Verlag, Berlin, 2008.



\bibitem[Pop]{Pop} Pop, F., Galoissche Kennzeichnung $p$-adisch abgeschlossener K\"{o}rper, J. reine angew. Math. 392 (1988), 145--175.



%%\bibitem[\text{Sa\"\i di-Tamagawa}]{Saidi-Tamagawa}Sa\"\i di, M. and Tamagawa, A., The $m$-step solvable anabelian geometry of number fields,  preprint,  arXiv:1909.08829.

%%\bibitem[Serre]{Serre} Serre, J.-P., Abelian $l$-adic representations and elliptic curves, Second edition, Advanced Book Classics, Addison-Wesley, Redwood City, 1989.

\bibitem[Shimizu]{Shimizu} Shimizu, R., The Neukirch-Uchida theorem with restricted ramification,  preprint,  arXiv:2009.10431.

\bibitem[Uchida]{Uchida} Uchida, K., Isomorphisms of Galois groups, J. Math. Soc. Japan, 28 (4) (1976), 617--620.

\end{thebibliography}
\end{document}